\documentclass[final]{dmtcs-episciences}

\usepackage[utf8]{inputenc}
\usepackage{amsmath}
\usepackage{color}
\usepackage{xypic}
\usepackage{graphicx}
\usepackage{subfigure}
\usepackage{float}
\usepackage{ifpdf}
\usepackage{url}
\usepackage{enumerate}
\usepackage{xspace}
\usepackage{multirow}
\usepackage{hhline}
\usepackage{microtype}
\usepackage{lmodern}
\usepackage{tabularx}

\makeatletter
\newcommand{\miniscule}{\@setfontsize\miniscule{4}{5}}
\makeatother

\usepackage{tikz}
\usetikzlibrary{calc,decorations.markings}

\tikzset{
unnumbered vertex/.style={draw, circle, fill=white, very thick, inner sep=1pt},
normal vertex/.style={draw, circle, fill=white, very thick, inner sep=0.5pt, font=\miniscule, minimum size=7pt},
normal edge/.style={very thick},
polygonal schema/.style={dashed, thin},
polygonal schema left/.style={polygonal schema, decoration={
	markings,
	mark=at position 0.5 with {
		\fill[solid, thin, black] (0.2,-0.2pt) -- (-0.2,0.2) -- (-0.2,-0.2pt) -- cycle;
	}
}, postaction={decorate}},
polygonal schema right/.style={polygonal schema, decoration={
	markings,
	mark=at position 0.5 with {
		\fill[solid, thin, black] (0.2,0.2pt) -- (-0.2,-0.2) -- (-0.2,0.2pt) -- cycle;
	}
}, postaction={decorate}}
}

\newcommand{\nx}{\ensuremath{\mathrm{nx}}\xspace}
\newcommand{\Mod}[1]{\ (\mathrm{mod}\ #1)}
\newcommand{\gen}{\ensuremath{\mathrm{mingen}}}

\newtheorem{theorem}{Theorem}
\newtheorem{lemma}[theorem]{Lemma}
\newtheorem{proposition}[theorem]{Proposition}
\newtheorem{corollary}[theorem]{Corollary}

\newtheorem{note}{Note}

\newtheorem{definition}{Definition}

\numberwithin{equation}{section}
\usepackage[round]{natbib}

\author{Gunnar Brinkmann\affiliationmark{1}
\and Thomas Tucker\affiliationmark{2}
\and Nico Van Cleemput\affiliationmark{1}  }

\title[On the genera of polyhedral embeddings of cubic graphs]{On the genera of polyhedral embeddings of cubic graphs.}

\affiliation{Ghent University, Belgium \\ Colgate University, USA}


\received{2020-08-24}

\revised{2021-04-07}

\accepted{2021-10-06}

\begin{document}

\publicationdetails{23}{2021}{3}{13}{6729}
\maketitle
\begin{abstract}
  In this article we present theoretical and computational results on the existence of polyhedral embeddings of graphs. The emphasis is on cubic graphs. We also
  describe an efficient algorithm to compute all polyhedral embeddings of a given cubic graph and constructions for cubic graphs with some special properties of
  their polyhedral embeddings.  Some key results are that even cubic graphs with a polyhedral embedding on the torus can also have polyhedral embeddings
  in arbitrarily high genus, in fact in a genus {\em close} to the theoretical maximum for that number of vertices, and that there is no bound on the number of genera in
  which a cubic graph can have a polyhedral embedding. While these results suggest a large variety of polyhedral embeddings, computations for up to 28 vertices suggest that
  by far most of the cubic graphs do not have a polyhedral embedding in any genus and that the ratio  of these graphs is increasing with the number of vertices.
\end{abstract}


\section{Introduction}

Three-dimensional convex polyhedra, or short polyhedra,
as geometric objects have already been a subject of mathematical research in ancient Greece and played an important role ever since. Since Steinitz's theorem
characterizing the graph formed by the edges of polyhedra as 3-connected plane graphs, these graphs -- also referred to as polyhedra --
are also studied as combinatorial objects. While the geometric characterization cannot directly be generalized to graphs on surfaces of higher genus, some
key properties of polyhedra can be used to generalize the concept to {\em polyhedral embeddings} of higher genus:

\begin{definition}
A polyhedral embedding of a graph $G=(V,E)$ in an orientable surface is an embedding so that
each facial walk is a simple cycle and an intersection of two faces is either empty,
a single vertex, or a single edge.
\end{definition}

An equivalent definition \cite{graphs_on_surfaces} is that $G$ is 3-connected and that the embedding has facewidth at
least 3 (with the facewidth of a plane graph defined as infinite). For cubic embedded graphs, this is equivalent to the dual being simple -- that is without multiedges or loops.
Like in most articles, in this article the word {\em graph} refers to simple graphs, but in some cases, like e.g.\ when talking about the dual, where the result of the operation might be a
multigraph, we emphasize the requirement of being simple.

Polyhedral embeddings share many properties of polyhedra.  For example, the dual of a polyhedral embedding is a polyhedral embedding \cite{Mohar1997}.
Furthermore, local
modifications of an embedding preserving symmetry,
such as next to dual also e.g.\  truncation, ambo, and many others preserve polyhedrality of an embedding \cite{Heidi_masterthesis}.  Finally, a generalization
of Whitney's theorem that polyhedra have a unique embedding in the plane, shows that for each genus $g$, there is a number
$\zeta (2g)$ such that no graph can have
more than $\zeta (2g)$ polyhedral embeddings in the orientable surface of genus $g$ \cite{num_embed_bound}.

Cubic graphs are an especially interesting class of graphs, as for many more general conjectures it has been shown that it is sufficient to prove them for cubic
graphs.  Furthermore the dual of a cubic polyhedrally embedded graph is a simple triangulation.  Triangulations of surfaces have a long history in topology and
combinatorics: the Hauptvermutung of the 1920s, the Ringel-Youngs Theorem of the 1960s, or e.g \cite{Malnic, Negami}.  On the other hand, most of that work does
not analyze the structure of the cubic dual embedding.

We deal with {\em combinatorial embeddings} in closed oriented surfaces. We will first describe the basic concepts and notations.
For further results and definitions, we refer the reader to \cite{top_graph_theory}. For vertices $x,y$ of an embedded graph $G$, an edge $\{x,y\}$
is interpreted as
two oppositely directed edges -- the edge $(x,y)$ directed from $x$ to $y$ and the edge $(y,x)$ directed from $y$ to $x$.
We refer to $(y,x)$ as the {\em inverse} edge of $(x,y)$ and write $e^{-1}$ for the
inverse of the directed edge $e$.
At each vertex we have a rotational order of the directed
edges starting at that vertex, that we interpret in figures as clockwise. Such an assignment of cyclic orders is often called a \textit{rotation system}. The {\em mirror image} of a combinatorial embedding assigns the reverse cyclic order for each vertex and represents the same topological embedding in the oppositely oriented surface.
An isomorphism between two embedded graphs is a graph isomorphism that also respects
the rotational order at each vertex or reverses the rotational order at each vertex, in which case we call the 
isomorphism {\em orientation reversing}.
When we talk about the {\em next}
(directed) edge of a directed edge $e$ (notation: $\nx(e)$), this refers to the next edge in rotational order around the start vertex of
$e$. A face is a cyclic sequence $e_0, \dots ,e_{k-1}$ of directed edges, so that for $0\le i <k$ we have
$\nx(e_i^{-1})=e_{(i+1) \Mod{k}}$. In figures this corresponds to a counterclockwise traversal of the edges.
We say that the set $\{e, \nx(e)\}$ of edges $e, \nx(e)$ forms an {\em angle}.
In this case $e^{-1}, \nx(e)$ follow each other
in a face of $G$ and $\nx(e)^{-1},e$ follow each other in a face of the mirror image.

If we have a directed simple cycle $C=(e_0, \dots ,e_{k-1})$ with
$e_i=(v_i,v_{(i+1) \Mod{k}})$ and a
directed edge $e=(v_{i+1},x)$ with
$e\not\in \{e_i^{-1},e_{(i+1) \Mod{k}}\}$,
then we say that $e$ is right of $C$ if in the rotational order after $e_{(i+1) \Mod{k}}$, the edge $e$
comes before $e_i^{-1}$. Otherwise we say that $e$ is left of $C$. An undirected cycle $C$ corresponds
to two directed cycles. If one of them is a face, we also call $C$ a facial cycle.
Let $v_i,v_{i+1},\dots ,v_k$
be a maximal subsequence of vertices of a directed cycle $C$ that all have no edges on the left (right) of $C$. If this sequence
contains all vertices, then $C$ ($C^{-1}$) is a face. Otherwise (assume w.l.o.g. $i>0$) we call
$(e_{i-1},e_i,\dots ,e_{k})$ a left (right) facial subpath of $C$.
In fact in that case $(e_{i-1},e_i,\dots ,e_{k})$, resp. $(e^{-1}_{k},\dots ,e^{-1}_{i-1})$ is part of a face.
This concept is visualized in Figure~\ref{fig:subpath}.

\begin{figure}[h!t]
	\centering
	\includegraphics[width=0.3\textwidth]{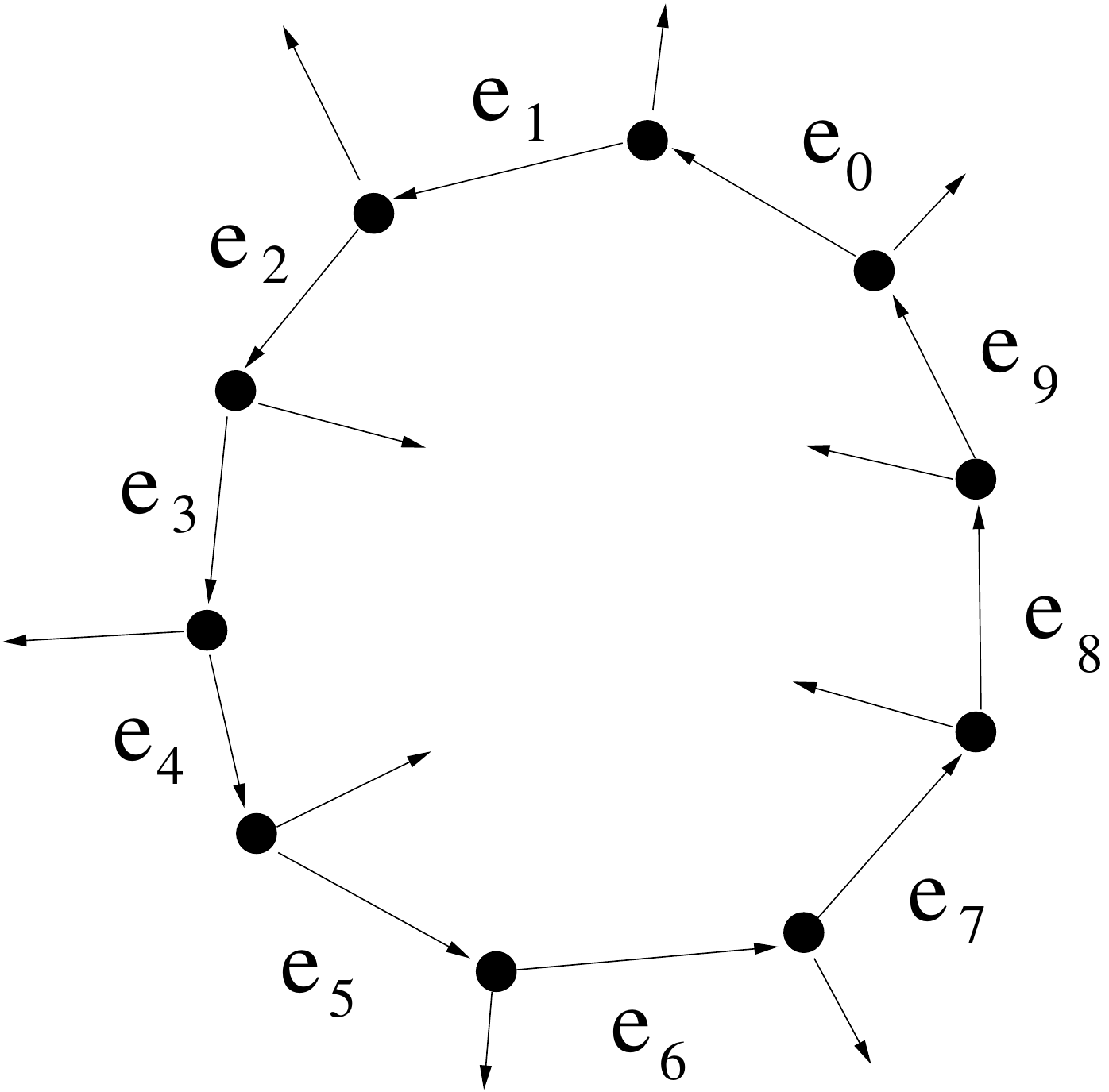}
	\caption{A directed cycle as subgraph in a cubic embedded graph with left facial subpaths $(e_9,e_0,e_1,e_2)$,
        $(e_3,e_4)$, $(e_5,e_6,e_7)$, and right facial subpaths $(e_7,e_8,e_9)$, $(e_2,e_3)$, and $(e_4,e_5)$.}
	\label{fig:subpath}
\end{figure}

Note that in a cubic graph the number of left and right facial subpaths of a directed cycle is always the same and that the first and last edge
of a left facial subpath always occur in right facial subpaths too (and the other way around). These are
the edges with one endpoint with a left edge and one with a right edge. A cycle $C$ in a graph $G$ is called {\em induced} if all edges of $G$
with both endpoints on $C$ are also edges of $C$.

In the following we are especially interested in the genus of polyhedral embeddings.
For an embedded graph, the genus is defined as the genus of the surface on which it is embedded.
If $G$ is an abstract, not embedded, graph, then $\gen(G)$ is the smallest genus of a surface on which $G$ can be embedded.
A genus embedding of a graph is an embedding in the smallest possible genus.
Of course polyhedra -- that is 3-connected plane graphs --
all have a polyhedral embedding in the plane.

Due to Whitney's famous theorem, plane embeddings of 3-connected
planar graphs are -- up to mirror images -- unique. In fact this is true
in the more general setting of polyhedral embeddings, as due to Theorem~8.1 in
\cite{Thomassen_90} stating that embeddings of 3-connected planar graphs in genus at least 1 have facewidth at most 2,
and therefore cannot be polyhedral embeddings.

\begin{theorem}\label{thm:polyhedra}
  Let G be a 3-connected planar graph.  Then the embedding in the plane is polyhedral and it is (up to mirror images)
  the unique plane embedding and also
  the unique polyhedral embedding.

  This implies that 
  there are no polyhedral embeddings of planar graphs, but genus embeddings of 3-connected planar graphs in the plane.
\end{theorem}

In Section~\ref{sec:further} we will show that already for cubic graphs $G$ with $\gen(G)=1$ even 
polyhedral embeddings with a genus close to the theoretical maximum for the
number of vertices exist and that there
are even cubic graphs that do have polyhedral embeddings, but not in minimum
genus.

\section{Basic results}

\begin{lemma}\label{lem:3out}
  Let $P_G$ be a polyhedral embedding of a graph $G$ and $C=(v_1,v_2,\dots ,v_k)$ a simple induced
  directed cycle of $G$.
Then $C$ ($C^{-1}$) either is a face or there are at least 3 vertices in $C$ that are starting points of an edge 
left (right) of $C$.
If $G$ is cubic and neither $C$ nor $C^{-1}$ is a face, then $C$ has at least 2 left and at least 2 right facial subpaths.
\end{lemma}

\begin{proof} 

  Vertices with edges left (right) of $C$ separate $C$ into segments between two such vertices
  that belong to the same face. If $C$ $(C^{-1})$ is not a face and we have one vertex with edges to the left (right), then we have one segment and that
would be part of a face with a facial walk that contains this vertex twice and is therefore
not simple. If we have two such vertices, we have two segments.
If they belong to different faces, we have two faces with the intersection containing these two vertices, which
are not part of the same edge as $C$ is induced. Otherwise we have one cycle where these two vertices occur multiple times
-- so the facial walk would not be simple.

If $G$ is cubic and $C$ had exactly one left facial subpath, it also had one right facial subpath and these paths would share
two edges, which is a contradiction to the embedding being polyhedral: if the subpaths are in the same face of the embedding, the face boundary would not be a cycle as two oppositely directed edges
are contained
and otherwise the two faces would share at least two edges.
\end{proof}

A direct consequence of this lemma is the following corollary, which also contains the results from Lemma~2.1, Lemma~2.2
and Lemma~2.3(a) from \cite{movo2006}:

\begin{corollary}\label{cor:smallfix}
Let $P_G$ be a polyhedral embedding of a cubic graph $G$ and $C$ be an undirected cycle that is a 3-cycle, a 4-cycle, or an induced 5-cycle
of $G$. Then -- except if $G=K_4$ and $C$ is a 4-cycle -- $C$ is a facial cycle.

If $C$ is an induced undirected 6-cycle, then $C$ is either a facial cycle or it has 3 edges to each side. Furthermore for both sets of three edges to one side, the starting vertices of the edges on $C$ do not form a path of length 2.
\end{corollary}

Note that if a 4-cycle in a cubic graph with more than 4 vertices is not induced, the graph has a 2-cut and therefore no
polyhedral embedding in any surface.

\section{An algorithm to compute polyhedral embeddings of cubic graphs}\label{sec:alg}

Already in \cite{movo2006} a computer program was used to check the
existence of polyhedral embeddings of some cubic graphs. At that time
Gr{\"u}nbaum's conjecture that there are no cubic graphs with chromatic
index 4 that admit a polyhedral embedding in an orientable surface was
still open and it was used to check all {\em weak snarks} on up to 30
vertices for the existence of a polyhedral embedding.  We use the term
{\em weak snark} for cyclically 4-connected cubic graphs with
chromatic index 4 and -- as usual -- the term {\em snark} if in
addition they have girth 5.  In \cite{movo2006} a straightforward
algorithm was used that constructed all $2^{|V|-1}$ embeddings and tested them
for being polyhedral. This approach is of course only suitable for
relatively small lists of not too large graphs and would -- even on a very large cluster -- not be feasible for the
graphs tested in this article, so that we had to develop a faster program, that we will refer to as {\em poly\_embed}. In the meantime the conjecture has
been refuted by Kochol \cite{polembsnark}, but as for Kochol's
counterexample it is neither known, whether it is smallest possible,
nor whether it has smallest possible genus, we also tested the weak
snarks on up to 36 vertices, which have been constructed in the
meantime \cite{snarkapp} and are available at \cite{HoG}.  Furthermore
we tested large lists of cubic graphs that are not
snarks to find examples for the constructions in
Section~\ref{sec:further}.

For a cubic graph $G$ 
we define two vertices of $G$ as being related, if they are contained in a common 3-, 4- or 5-cycle
and begin by computing the equivalence classes of the transitive closure of this relation.

\begin{proposition}\label{note:equclass}

  If $v,w$ are vertices in the same equivalence class of a cubic graph
  $G$, then in any polyhedral embedding of $G$ the order of edges
  around $w$ is uniquely determined by the order around $v$.
  
\end{proposition}

\begin{proof}
  For trivial equivalence classes the result is immediate, so assume
  that $f$ is a 3-, 4- or 5-cycle in $G$ containing the edges
  $\{x,v\}$ and $\{y,v\}$. Then fixing the rotation around $v$ fixes
  the order in which $f$ is traversed,
  w.l.o.g.\ $(v,x)=\nx((v,y))$. Due to Corollary~\ref{cor:smallfix}
  this implies that the rotation around $x$ is fixed, as the second
  edge incident with $x$ in the 5-cycle must be the next in the
  rotational order after $(x,v)$. The rest of the proof follows by
  transitivity.
 
  \end{proof}

Note that for the transitive closure it makes no difference whether in
the initial step you choose all 5-cycles or just the induced
ones, as non-induced ones are the union of a 3- and a 4-cycle. Depending on the graph, there can be just one equivalence class
in the best case, but  if the girth is at least 6, each vertex forms a
separate equivalence class.  Having
computed the equivalence classes, we compute the set of
compatible rotations. Already during this computation
there can be conflicts implying that no polyhedral embeddings exist --
see Figure~\ref{fig:contradict}. If in the recursion the rotation of a vertex in an equivalence class is switched, the rotations of
all vertices in the class are switched simultaneously.

\begin{figure}[h!t]
	\centering
	\includegraphics[width=0.3\textwidth]{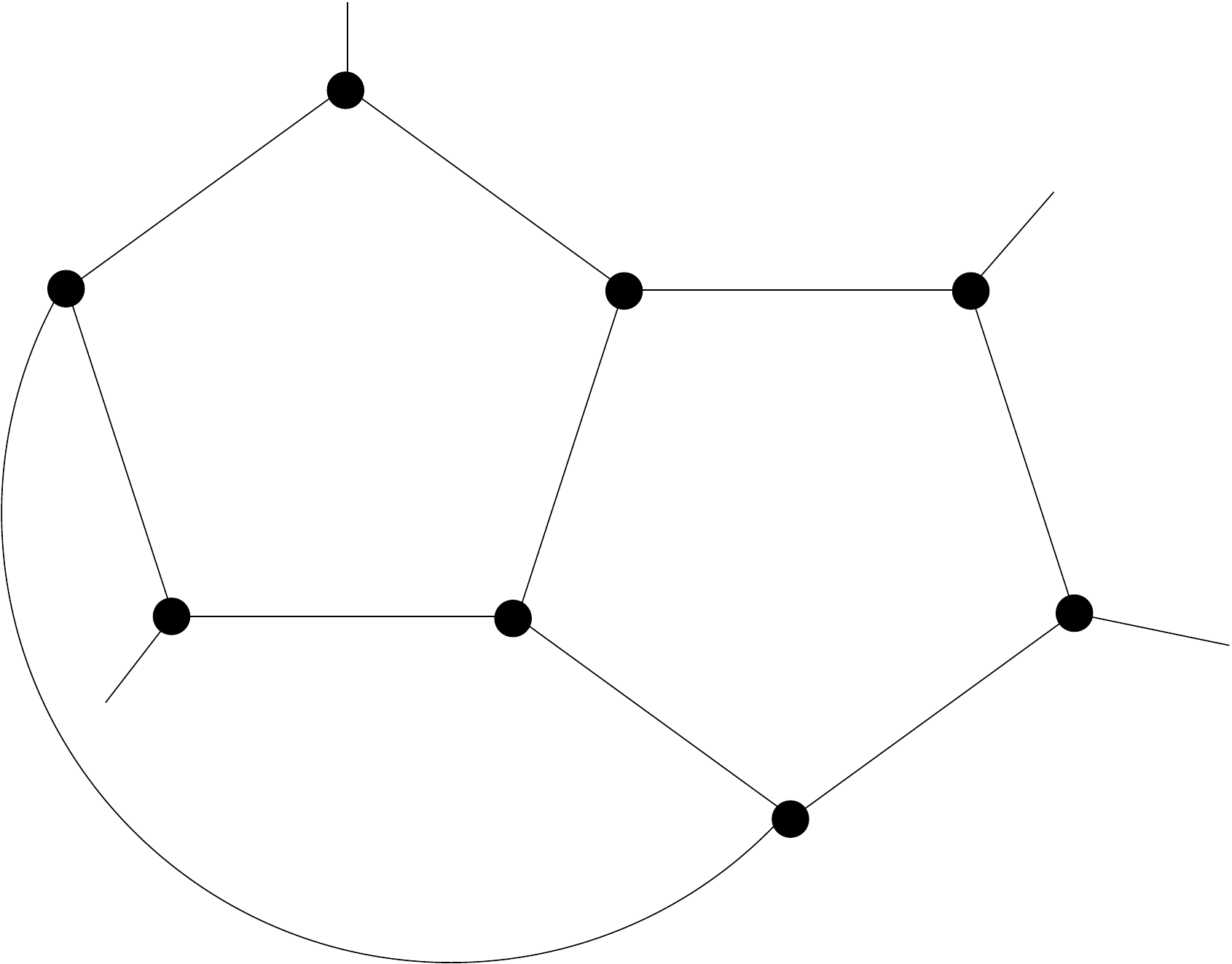}
	\caption{An example of two 5-cycles and a 4-cycle that cannot be embedded as 5-gons and 4-gons.}
	\label{fig:contradict}
\end{figure}

The algorithm then starts by fixing the rotation of the largest
equivalence class and assigning an initial rotation to all other
classes. In the following branch and bound algorithm, the rotations
of equivalence classes are recursively switched in order to enumerate
all possible combinations.  For graphs with girth 6 this does not help
at all, as each equivalence class is just a single vertex, but even if
the graph has just two 5-cycles sharing an edge, the number of
possible rotations to be considered is already decreased by a
factor of 128. The recursion processes the equivalence
classes in decreasing order of their size. Once a rotation is
chosen, the rotation of the vertices in this class is considered to
be fixed until the recursion returns to this class. So at each
recursion depth we have a set of {\em fixed} vertices. As their
rotation will not be modified at higher depths of the recursion, a
contradiction within this set -- e.g. a face with only fixed vertices
and without a simple cycle as boundary -- allows to bound the
recursion and backtrack.

We use the following observation 
when testing for polyhedral embeddings. It is based on the fact that in a cubic graph two cycles sharing a vertex share a whole edge.

\begin{note}

  An embedding of a cubic graph is polyhedral if and only if
  \begin{itemize}
  \item no vertex occurs twice in a facial walk and
  \item no two faces share more than two vertices.
    \end{itemize}

\end{note}

The advantage of this characterization from a computational point of view is that the second
criterion, which is a criterion about all pairs of faces, is reduced
to a very simple test. If the faces are implemented as bitvectors, it
is sufficient to compute the intersection -- a very simple and fast
operation for sets represented as bitvectors -- and afterwards test whether the intersection contains
more than two bits -- again very fast. For graphs that are not cubic,
an extra test would be needed for the case of two common vertices in order to test
whether they are the endpoints of a common edge lying in both faces.

During the recursion, the algorithm distinguishes between the case when
the next equivalence class is a single vertex or not.
As long as
there are nontrivial equivalence classes, that are not fixed, they are processed in
decreasing order of their size. When there are only trivial equivalence classes
left, we use the concept of {\em obstruction}: an obstruction in an
embedded graph is a part of the boundary between two occurences of the
same vertex or the union of the vertex sets of two faces which share at
least 3 vertices.

For trivial equivalence classes, at each recursion depth, a set
$S=v_1,v_2,\dots ,v_k$ is chosen and for each $1\le i\le k$ the
vertices $v_1,\dots ,v_{i-1}$ are declared fixed (with the rotation
they initially have) and vertex $v_i$ is declared fixed after the
rotation was switched. For each $i$ then the next recursion depth
is called with the vertices $v_{i+1}, \dots ,v_k$ being not fixed and available for the next set.
If there is no
obstruction, the embedding is polyhedral and output (or
counted if the options require no output) and $S$ is chosen as the set of all vertices that are not yet fixed.
Otherwise $S$ is chosen as the set of not fixed vertices in an obstruction.

After each switching of a rotation, Corollary~\ref{cor:smallfix} is applied to all hexagons.
This might lead to a contradiction or to uniquely determined rotations for some not yet fixed vertices in these hexagons.
If no contradiction is found, a -- preferably small -- obstruction is searched. If there is none, the embedding is polyhedral
and if there is one with only fixed vertices, we can backtrack.

In fact it turned out to be faster not to do a complete search for
obstructions with only fixed vertices -- except at the end of the
recursion or very close to the root of the recursion.

\bigskip
\textbf{Testing the program:} In order to test the program we wrote
two other programs. A first program creates all possible embeddings of
a given graph simply by writing all possible combinations of orders
around each vertex. We then independently wrote a program to check
whether a given embedded graph is polyhedral by checking that each
face is a cycle and checking the intersection of each pair of
faces. We compared the number of polyhedral embeddings and their genus
for all cubic graphs on up to 22 vertices.
For 24 vertices only the cubic graphs with girth at least 4 were independently tested. We had agreement in all
cases.

\bigskip
\textbf{Performance:} All times are given on a linux laptop
with Core i7-6820HQ CPU, 2.70GHz (running typically at 3.2 Ghz).
Constructing all polyhedral embeddings of all
117,940,535 cubic graphs on 24 vertices takes 14 minutes, so in
average 140,000 graphs per second are processed. The speed is for a
large part due to the presence of small cycles exploited in the first
step of the algorithm. Restricting the tests to the cubic graphs on 24
vertices with girth at least 6, it takes 6.4 seconds for 7,574 graphs
-- so only around 1,200 per second. Testing all 153,863 weak snarks  on up to 30 vertices
takes about 0.55 seconds. Testing all 432,105,682 weak snarks on up to 36 vertices
takes a bit less than half an hour. Constructing all polyhedral embeddings of the Kochol snark on 74
vertices takes 0.03 seconds. Note that the Kochol snark has many pentagons.

\medskip

\subsection{Computational results}
~\\
In \cite{polembsnark} Kochol refuted Gr\"unbaum's conjecture that cubic
graphs that have a polyhedral embedding also have an edge-3-colouring
\cite{gruenbaumconj}, by constructing a snark with 74 vertices and a polyhedral embedding. It has
genus 5 and the (unique) genus embedding is polyhedral and 
in fact also the only polyhedral embedding. The fact that
the genus embedding is unique is -- just like all other genera in this
paper -- computed by the program {\em multi\_genus} described in \cite{genuscomp}. It is
neither known whether Kochol's graph is the smallest counterexample nor whether
there are counterexamples with smaller genus. So we tested all
available smaller snarks and weak snarks in the database HoG
\cite{HoG}, that is all weak snarks on up to 36 vertices, an
incomplete list of 19,775,768 snarks on 38 vertices and 25 snarks with
girth 6 on 40 vertices. None of them allows a polyhedral embedding.

Of course snarks are non-hamiltonian, but having no hamiltonian cycle is a much weaker condition than being a snark, so a question is when the first
non-hamiltonian cubic graphs with a polyhedral embedding occur.
The smallest such graph is the Coxeter graph, which has 28 vertices. It is in fact the only cubic nonhamiltonian graph on up to 28 vertices with a polyhedral embedding.
The Coxeter graph has genus 3 and 16 embeddings in genus 3, which are all isomorphic. The embeddings in genus 3 are the only polyhedral ones.

In order to get an idea of how common polyhedral embeddings of cubic
graphs are, but also in order to search for small examples of cubic graphs
with several polyhedral embeddings, polyhedral embeddings in more than
one genus and even graphs with polyhedral embeddings, but not in their
minimum genus, we tested all cubic graphs on up to 28 vertices for polyhedral embeddings. The lists of
graphs were generated by the program {\em minibaum} described in \cite{minibaum}. The results
are given in Table~\ref{tab:all} and some special embeddings found are
given and used for the construction of infinite sequences in
Section~\ref{sec:further}.

\begin{table}
  	\centering
\begin{tabular}{l|c|c|c|c}
  number          & number    & graphs    & graphs with  & graphs with pol.  \\
  \quad of      & of graphs  &   with pol. emb.  & more than  &  emb. in more \\
  vertices         &                         &                       &    one pol. emb. &  than one genus \\
    \hline
  4  & 1  &  1( 100 $\%$) & 0  & 0  \\
  6 & 2  & 1 ( 50 $\%$) &  0 &  0 \\
  8 & 5  & 2 ( 40 $\%$) & 0  & 0  \\
   10 & 19  & 5 (26 $\%$)  &  0 &  0 \\
   12 & 85  & 14 (16.5 $\%$) & 0  &  0 \\
    14   & 509  & 51  (10 $\%$)& 1  & 0  \\
   16 & 4,060  & 240 (5.9 $\%$) & 4  & 0  \\
  18 &  41,301 &  1,349 (3.3 $\%$) &  28 & 0  \\
  20 & 510,489  & 9,464 (1.9 $\%$)  & 278  & 0  \\
     22  &  7,319,447 & 84,230 ( 1.2 $\%$)  & 2,997  &  0 \\
   24 & 117,940,535  & 909,431 ( 0.77 $\%$) &  32,438 &  0 \\
   26  & 2,094,480,864  & 10,902,162(0.52 $\%$)  & 348,078  &  4 \\
   28  & 40,497,138,011  & 138,008,652 (0.34$\%$)  &  3,909,031 &  157 \\
\end{tabular}
\medskip
\caption{\label{tab:all} Computational results on all cubic graphs with a given number of vertices. Note that {\em ``with more than one polyhedral embedding''} refers to different combinatorial embeddings in the sense of being not identical or mirror images. It does not necessarily refer to non-isomorphic embeddings. }
\end{table}

In Table~\ref{tab:allpolembed} numbers of non-isomorphic polyhedrally
embedded cubic graphs are given. They were obtained by filtering the
output of {\em poly\_embed} for
non-isomorphic embeddings and computing the genus of the underlying
graphs by {\em multi\_genus}. Although even for
28 vertices the ratio of polyhedral embeddings that are not in minimum
genus is still only $0.00016 \%$, in the last step it grows much faster than the number
of polyhedral embeddings. The fact that the same graph embedded in a small genus has more, and
in average smaller, faces than when embedded in a larger genus,
suggests that polyhedral
embeddings should more often occur in or close to the smallest genus
that the graph can be embedded in. Though the computed numbers seem to
support such a claim, the numbers are much too small to draw well
grounded conclusions.

\begin{table}
  	\centering
\begin{tabular}{l|c|c|c}
  number          & number    & number of   & number of \\
  \quad of      & of pol.  &   min genus  &   not min genus  \\
  vertices         &   embeddings   &   pol. emb.    &   pol. emb.  \\
    \hline
  4  & 1 &  1 & 0  \\
  6 & 1&  1 & 0 \\
  8 &  2  & 2 & 0 \\
   10 &  5  & 5 & 0   \\
   12 &  14  & 14 & 0   \\
    14   &  51  &  51 & 0   \\
   16 &  240  & 240 & 0   \\
  18 &  1,361   &  1,361& 0   \\
  20 &  9,704  & 9,704 & 0    \\
     22  &  87,433   & 87,433 & 0    \\
   24 &  946,083  & 946,083 & 0    \\
   26  &  11,298,676  & 11,298,671 & 5 \\
   28  &  142,414,959  & 142,414,731 &  228\\
\end{tabular}
\medskip
\caption{\label{tab:allpolembed}  Computational results on all non-isomorphic polyhedrally embedded cubic graphs with a given number of vertices. Here,  {\em non-isomorphic} refers to the concept of isomorphism of embedded graphs -- not of abstract graphs. }
\end{table}

\section{Constructions and further results on polyhedral embeddings}\label{sec:further}

As Theorem~\ref{thm:polyhedra} states, polyhedral embeddings of 
3-connected planar graphs are unique and always in the plane -- which is of course also the minimum genus in which the graph can be embedded.
This behaviour changes dramatically
once we move away from planar graphs. Already for genus $1$ there are
cubic graphs $G$ with $\gen(G)=1$ that have polyhedral embeddings of different genera as well as such graphs with polyhedral embeddings -- but not in minimum genus.

As the number of faces decreases when the genus increases, there is a
natural upper bound for the genus of a polyhedral embedding with a given number of vertices and edges:

\begin{lemma}
  If $P$ is a polyhedral embedding of a cubic graph with $n$ vertices, then the genus $g(P)$ fulfills

  \[ g(P) \le  \frac{n-\sqrt{12n+1}+3}{4}  \]

  This bound is sharp for infinitely many $P$. To be exact: for each $n$ for which the right hand side of the formula is an integer number, there
  is a polyhedral embedding $P$ of a cubic graph with $n$ vertices and genus $g(P) =  \frac{n-\sqrt{12n+1}+3}{4}$.

\end{lemma}

\begin{proof}

  An embedding of a cubic graph is polyhedral if and only if the dual
  is a simple graph. As the dual graph also has $\frac{3n}{2}$ edges
  and a simple graph with $f$ vertices has at most $\frac{f(f-1)}{2}$
  edges, we get $f\ge \sqrt{3n+0.25}+0.5$ and inserting this into the
  Euler formula and the formula for the genus we get the result in the
  lemma.


  If $\frac{n}{4} - \frac{\sqrt{3n+0.25}}{2} +\frac{3}{4}$ is an integer number, replacing $n$ by $\frac{f(f-1)}{3}$ and simplifying the term gives that 
  this is the case if and only if $\frac{(f-3)(f-4)}{12}$ is an integer. This can easily be seen to be the case if and only if $f \Mod {12} \in \{0,3,4,7\}$ -- allowing
  a triangular embedding of the complete graph $K_f$  (see  \cite{top_graph_theory}, \cite{genus_complete_graph}, or \cite{facedistributions})
  with the given genus, which has a simple cubic dual with the given number $n$ of vertices. 
  
\end{proof}

Note that the dual cubic graph of a triangular embedding of a complete
graph is polyhedrally embedded in the same genus as the complete
graph, but that this is not necessarily a minimum genus embedding. In
fact tests for duals of some embeddings of $K_{12}$ gave cubic graphs
with $44$ vertices and genus 4 or 5 (instead of 6 as $K_{12}$), but
the only polyhedral embeddings were in genus 6. We did not do a
complete enumeration of all possible embeddings of $K_{12}$.

In Theorem 8.2 of \cite{Thomassen_90} an infinite sequence of toroidal
cubic graphs is given that -- with $n$ the number of vertices -- have an embedding with
facewidth at least $4$ in genus $\frac{n}{8}-2\sqrt{\frac{n}{8}}+1$,
so up to minor terms about half the genus of the maximum possible
genus. If one only requires that the embedding is polyhedral, we can
-- up to minor terms -- even get the same value as the maximum
possible genus. We first need the following lemma:

\begin{lemma}\label{lem:2hextorus}

  Let $G$ be the graph of a hexagonal tiling $T$ of the torus where all cycles of length at most 6 are faces.
  Then $G$ has (up to mirror images) at most 2 polyhedral embeddings.

\end{lemma}

\begin{proof}

  The toroidal embedding $T$, its mirror image $T^{-1}$, as well as
  the two embeddings $T_a, T_a^{-1}$ (one a mirror image of the other) where in each
  hexagon the rotations are alternatingly like in $T$ and like in
  $T^{-1}$ might be polyhedral. Note that if in $T_a$ or $T_a^{-1}$ the rotations in one original hexagon are given, they are uniquely determined for all others.

  If $T'$ is a polyhedral embedding of $G$ that is none of these, then
  -- due to Corollary~\ref{cor:smallfix} -- there is a hexagon where
  at three vertices the rotation is like in $T$ and at three vertices
  the rotation is like in $T^{-1}$ and as $T'$ is neither $T_a$ nor
  $T_a^{-1}$, the hexagon can also be chosen in a way that at two
  neighbouring vertices the rotation is like in $T^{-1}$. In
  Figure~\ref{fig:hex2} the two possible distributions of rotations in
  $T'$ for the two hexagons of $T$ sharing this edge are
  depicted. Parts of two faces of $T'$ are displayed by bold
  resp.\ dotted lines. As these two faces share at least two edges,
  $T'$ cannot be a polyhedral embedding. The condition that there are
  no short non-facial cycles is needed to guarantee that the
  two edges in Figure~\ref{fig:hex2} leading to the contradiction are distinct.

\begin{figure}[h!t]
	\centering
	\includegraphics[width=0.4\textwidth]{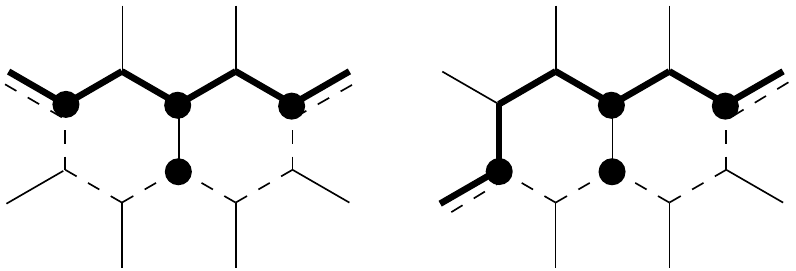}
	\caption{Portions of the hexagonal embeddings $T$ and $T'$ on the torus. $T$ is described by choosing the rotation as given by the normal clockwise
          traversal of edges around all vertices.  For $T'$ the rotation is reversed -- that is: like in $T^{-1}$ -- for bold vertices and like in $T$ for the
          other vertices.  Parts of two faces of $T'$ are marked with bold resp.\ dotted lines.}
	\label{fig:hex2}
\end{figure}

\end{proof}

\begin{lemma}\label{lem:hextorus}

  There are infinitely many $n$, so that there is a cubic graph with $n$ vertices and with
  a polyhedral embedding on the torus as well as on a surface of genus $\frac{n}{4} -\frac{3}{2}\sqrt{\frac{n}{2}}+1$
  -- and no other polyhedral embedding.

\end{lemma}

\begin{figure}[h!t]
	\centering
	\includegraphics[width=0.35\textwidth]{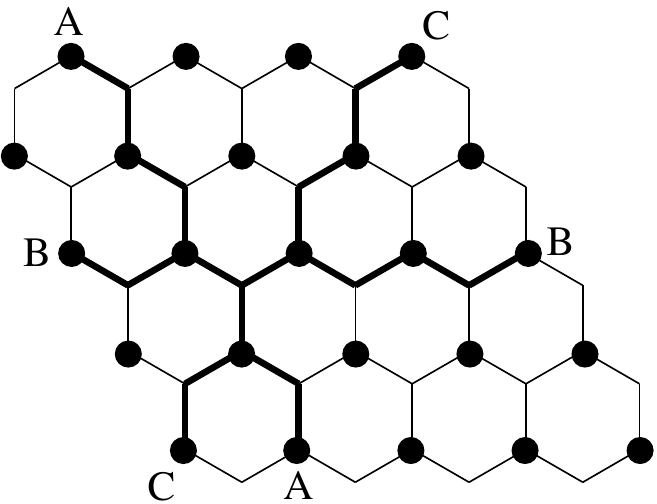}
	\caption{Two embeddings of a cubic graph: a hexagonal embedding $T$ on the torus with 16 hexagons is given by the rotation defined by the clockwise
          traversal of edges around all vertices.  A second embedding is described by defining the rotation around bold vertices to be reversed (that is: coming
          from the counterclockwise traversal of the edges) and around the other vertices like in $T$. Some facial cycles of the second embedding, which is
          in a higher genus are marked with fat lines.}
	\label{fig:hex}
\end{figure}

\begin{proof}

  Let $n$ be any number, so that $n=2k^2$ for some $k\ge 3$. Let $T$ be the cubic bipartite polyhedral embedding on the torus, formed by a $k\times k$
  parallelogram of $k^2$ hexagons, depicted for $k=4$ in Figure~\ref{fig:hex} with the sides identified like indicated by the letters in the figure.  This cubic
  graph has $n=2k^2$ vertices and $3k^2$ edges.  If we switch the rotation at each vertex of one of the bipartition classes, the faces of this new graph are
  zig-zag walks in the old one, also known as Petrie walks in a more general context. They are depicted in bold in Figure~\ref{fig:hex} for three closed Petrie
  walks.  Each face is a simple cycle and has $2k$ edges. As each edge is in two faces this gives $3k$ faces and therefore a genus of $\frac{k^2-3k+2}{2}$.
  Replacing $k$ by $\sqrt{\frac{n}{2}}$, we get the result in the lemma. It is easy to see that two faces either share no edges -- if they correspond to
  parallel Petrie walks -- or one edge, so that the embedding is polyhedral.

  The last part is a direct consequence of Lemma~\ref{lem:2hextorus} and that for $k=3$ (so $n=18$) -- where  Lemma~\ref{lem:2hextorus} cannot be applied --
  we get $\frac{n}{4} -\frac{3}{2}\sqrt{\frac{n}{2}}+1=1$.
  
\end{proof}

For a cubic graph $G$ and a genus $g$ let $n_g(G)$ denote the number of different polyhedral embeddings of $G$ in genus $g$ with mirror images identified
(or equivalently: the rotation around one vertex fixed). When referring to {\em different} polyhedral embeddings, we always assume that mirror images are considered
to represent the same embedding.

Note that the Heawood graph $H$ on the torus -- that is: the dual of $K_7$ embedded on the torus -- has many 6-cycles
that are not faces. It is a cubic hexagonal tiling and has (up to
mirror images) 8 different polyhedral embeddings on the torus, so $n_1(H)=8$. These 8 embeddings
are isomorphic. We will now prove the existence of some cubic graphs with special embedding properties:

\begin{theorem}\label{thm:existence}

  For each $k>0$  there is 
  
  \begin{description}
\item[(a)] a cubic graph $G$ admitting two polyhedral embeddings $P$, $P'$ so that the difference between the genera of these embeddings is at least $k$ and there is no embedding for any genus in between.

\item[(b)] a cubic graph $G$ admitting polyhedral embeddings in at least $k$ genera.

\item[(c)] a cubic graph $G$ and a genus $g$, so that $n_g(G)\ge k$. In fact for genus $g$ there is a graph $G$ with $n_g(G)\ge 8^g$.
  This is also true if one requires the embeddings to be non-isomorphic.

  \item[(d)] a cubic graph $G$, so that the smallest genus of a polyhedral embedding of $G$ differs from $\gen(G)$ by at least $k$.

    \end{description}

\end{theorem}

In fact (a) is a direct consequence of Lemma~\ref{lem:hextorus} and mentioned here only for completeness. In order to prove the other items, we will now first present
an operation on graphs that preserves some properties with respect to polyhedral embeddings. Then we will present some graphs with special properties of their polyhedral embeddings, which
we will use as building blocks for this operation.

\medskip

\subsection{Combining graphs}
~\\ The following construction for nontrivial 3-cuts has already been
used in several papers under various names and to study various
invariants. It was e.g. called {\em marriage} in \cite{kotzig_mariage}
when studying Hamiltonian cycles or {\em star product} in
\cite{movo2006} when studying edge colourings
of cubic graphs in the context of polyhedral embeddings.

For a graph $G$ with a vertex $v$ let $G_v$ denote the graph $G$ with the vertex $v$ deleted.

Let $G, G'$ be graphs with cubic vertices $v$ in $G$ and $v'$ in $G'$ and let the neighbours of $v$ be $a,b,c$ and those of $v'$ be $a',b',c'$.
Then $G_v \star_{(a,b,c,a',b',c')} G'_{v'}$ denotes a graph obtained by connecting $G'_{v'}$ to $G_v$ by adding the edges 
$\{a,a'\},\{b,b'\},\{c,c'\}$.
If $G,G'$ are embedded, the edges are inserted in the rotation at the same place where the edges from $v$, resp.\ $v'$ were located.
In cases where the exact identification is not important or clear from the context, we also write shortly $G_v \star G'_{v'}$ to simplify notation.
If $G,G'$ are embedded with genus $g$, resp.\ $g'$, in $G_v$ the vertices $a,b,c$ occur in the order $a,b,c$ and in $G'_{v'}$ the vertices $a',b',c'$ occur in the order $a',c',b'$, the fact
that  $G_v \star_{(a,b,c,a',b',c')} G'_{v'}$ is embedded with genus
$g+g'$ can now be
  obtained by the fact that (with $v()$ the number of vertices, $e()$ the number of edges and $f()$ the number of faces of a graph)
  $v(G_v \star G'_{v'})=v(G)+v(G')-2$, $e(G_v \star G'_{v'})=e(G)+e(G')-3$, and $f(G_v \star G'_{v'})=f(G)+f(G')-3$. As minimum genus embeddings of
  $G$ and $G'$ can be chosen in a way that the rotations around $v$ and $v'$ imply this order, we get $\gen(G_v \star G'_{v'})\le \gen(G)+\gen(G')$ for all possible
  $(a,b,c,a',b',c')$.

Lemma~\ref{lem:3edgecut} is essentially Theorem~3.1 from
\cite{movo2006}, but as we need it in a slightly different formulation
and also allow non-cubic graphs, we will just reprove it here:

We will need the following proposition also following from the proofs of Theorem~3.1, Theorem~3.3 and Theorem~3.5 in \cite{movo2006}.
An edge cut $S$ is a set of edges, so that there exists a bipartition of the vertex set of the graph, so that $S$ is the set of edges with endpoints in both sets.

\begin{proposition}\label{note:3cut}

  For $3\le k\le 5$ let $e_1,e_2,\dots,e_k$ be a $k$-edge-cut in a polyhedral embedding of a (not necessarily cubic) graph $G$.
  Then the subgraph of the dual graph whose vertices are the faces containing an edge in $\{e_1,e_2,\dots,e_k\}$ and whose edges correspond to $e_1,e_2,\dots,e_k$ is a cycle.
  Furthermore: if $C$ is a component of $G-\{e_1,e_2,\dots,e_k\}$ with the embedding induced by $G$,
  then all endpoints of  $e_1,e_2,\dots,e_k$ in $C$ are in the same face of $C$.

  \end{proposition}

\begin{proof}

  Each of the faces $f_1,\dots ,f_j$ containing at least one of
  $e_1,e_2,\dots,e_k$ contains an even number of these edges, as
  otherwise an edge has two endpoints in the same bipartition class.
  Let $D$ be the graph with vertices $f_1,\dots ,f_j$ and edges
  $\{f_i,f_h\}$ for all $f_i,f_h$ that share one of
  $e_1,e_2,\dots,e_k$. Due to polyhedrality $D$ is a simple graph.  As
  there are at most 5 edges, we have immediately that the vertex
  degrees in $D$ are only 2 or 4 and it is easy to see that vertices
  with degree 4 cannot exist. So $D$ is a simple cycle (so also
  $j=k$).

  This means that each $f_i\in \{f_1,\dots ,f_k\}$ has in each of the two components containing vertices of $f_i$ exactly one facial subpath or only a single vertex
  and the union of subpaths in the same component forms a closed facial walk after the removal of $e_1,\dots ,e_k$.

\end{proof}

\begin{lemma}\label{lem:3edgecut}

  Let $G$ be a graph with a polyhedral embedding in genus $g$ and a nontrivial 3-edge-cut.

  Then there are (smaller) embedded graphs $G',G''$ with cubic vertices $v'$ in $G'$ and  $v''$ in $G''$ and polyhedral embeddings in genus $g'$, resp.\ $g''$,
  so that $g'+g''=g$ and $G=G'_{v'} \star G''_{v''}$. If $G$ is cubic, then also $G'$ and $G''$ are cubic.

\end{lemma}

\begin{proof}
  As $G$ is 3-edge-connected, after the removal of the edge cut we
  have exactly two (embedded) components and due to
  Proposition~\ref{note:3cut} both have the three distinct (as $G$ is
  3-connected) vertices where an edge was removed in the same face. So
  we can add vertices $v'$, resp.\ $v''$ inside the face to form $G'$
  and $G''$. The fact that $g'+g''=g$ can again be shown by simply
  counting vertices, edges and faces and the fact that $G'$ and $G''$
  are polyhedrally embedded can be easily seen by checking the face
  adjacencies around the new vertices.
  
\end{proof}

\begin{lemma}\label{lem:extension}

  Let $G, G'$ be graphs with cubic vertices $v$ in $G$ and $v'$ in $G'$. 

  Then for any genus $g$ we have with $ 0\le j,j' \le g$

  \[ n_g(G_v \star G'_{v'})=\sum_{j+j'=g}(n_j(G)\cdot n_{j'}(G')). \] 

\end{lemma}

\begin{figure}
  	\centering
\includegraphics[width=0.8\textwidth]{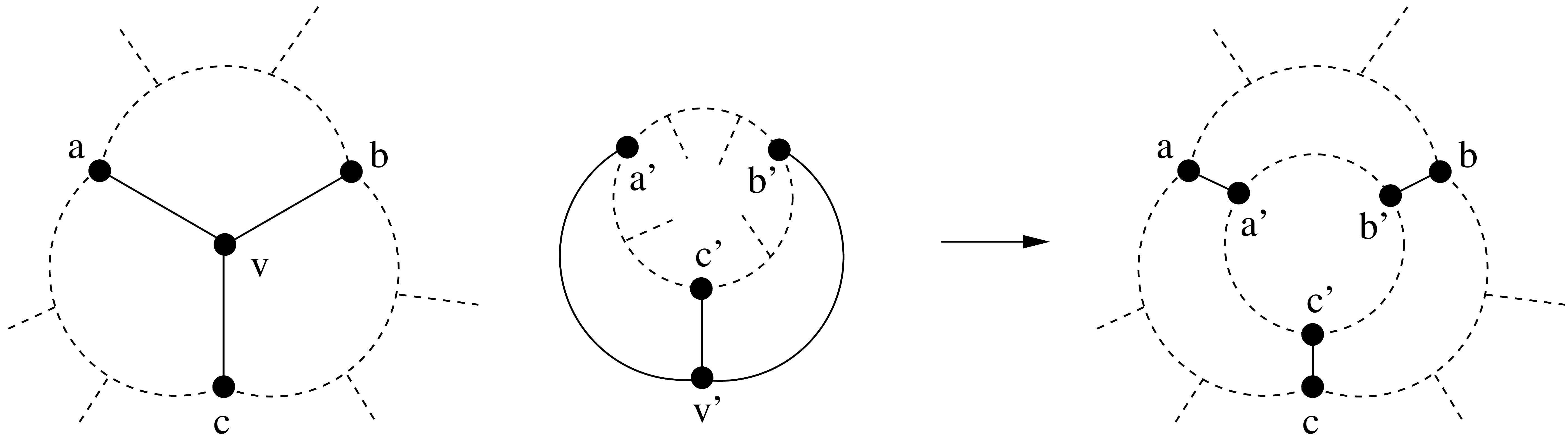}
\caption{The operation $G_v \star G'_{v'}$.}\label{fig:extend}
\end{figure}

\begin{proof}

  Let $a,b,c$, resp.\ $a',b',c'$ be the neighbours of $v$ in $G$,
  resp.\ $v'$ in $G'_{v'}$, named in a way that the edges $\{a,a'\},\{b,b'\}$
  and $\{c,c'\}$ are inserted when forming $G_v \star G'_{v'}$.

  Let now $j,j'$ be so that $0\le j,j' \le g$ and
  $j+j'=g$. Furthermore let $I_G$ be a polyhedral embedding of $G$ of
  genus $j$ with the rotation around $v$ equal to $a,b,c$ and let
  $I_{G'}$ be a polyhedral embedding of $G'$ of genus $j'$ with the
  rotation around $v'$ equal to $a',c',b'$. Inserting the edges
  $\{a,a'\},\{b,b'\}$ and $\{c,c'\}$ in the rotational order around
  $a,b,c,a',b',c'$ where the deleted edges in $I_G$ and $I_{G'}$ have
  been -- see Figure~\ref{fig:extend} -- we get an embedding of $G_v  \star G'_{v'}$
  with genus $j+j'=g$ and by checking face adjacencies
  it is easily seen to be polyhedral.

  This way all polyhedral embeddings of $G$ with genus $j$ can be
  combined with all polyhedral embeddings of $G'$ with genus $j'$ and
  doing this for all possible $j,j'$ we get $n_g(G_v \star G'_{v'})\ge \sum_{j+j'=g}(n_j(G)\cdot n_{j'}(G'))$. Note that two different
  embeddings for $G$ (and analogously $G'$) lead to different
  embeddings of $G_v \star G'_{v'}$, as the rotation around $v$ is
  always the same, so the difference is at vertices also present in
  $G_v \star G'_{v'}$.

  What remains to be shown is that each polyhedral embedding of $G_v  \star G'_{v'}$ is of this form.
Assume a polyhedral embedding of genus $g$ of $G_v  \star G'_{v'}$ to be given. Choose that of the two equivalent embeddings
  where in the face containing the edges $\{a,a'\},\{b,b'\}$ the facial walk is $b',b,\dots ,a,a',\dots$ (in the mirror image it is $b,b',\dots ,a',a,\dots$).
Then in the face containing  $\{b,b'\},\{c,c'\}$   the facial walk is $c',c,\dots ,b,b',\dots$ and we get that in the component containing $a,b,c$ the order in
  the face is $a,c,b$, so that if we insert a vertex into the face and connect it to the three vertices by assigning the reverse order, we get an embedding
  of $G$ with the order around the new vertex $a,b,c$. Analogously we get an embedding of $G'$ with rotational order around the new vertex $a',c',b'$, so that
  the embedding is in fact of the form counted in the formula.
  
  \end{proof}

The following corollary is a direct consequence of Lemma~\ref{lem:extension} and Theorem~\ref{thm:polyhedra}, which implies that for 3-connected planar graphs $G'$ we have
$n_{0}(G')=1$  and $n_{j'}(G')=0$ for $j'>0$.

\begin{corollary}\label{cor:equalnumber}

  Let $G, G'$ be graphs with cubic vertices $v$ in $G$ and $v'$ in $G'$ and let $G'$ be a 3-connected planar graph.

  Then for all $g$ we have  $ n_g(G_v \star G'_{v'})= n_g(G)$.

\end{corollary}

\begin{figure}
  	\centering
\includegraphics[width=0.9\textwidth]{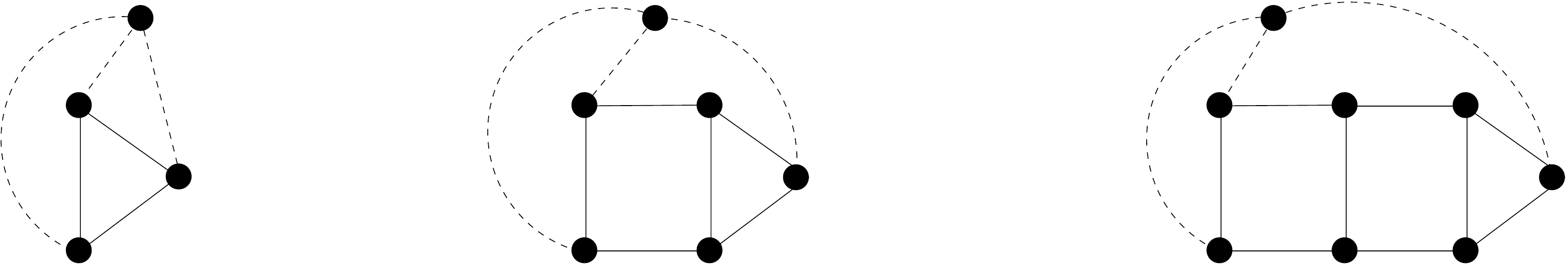}
\caption{3-connected plane graphs, so that after the removal of a vertex only pairwise non-isomorphic graphs with only triangles and quadrangles as bounded faces remain. }\label{fig:triangle2}
\end{figure}

\begin{lemma}\label{lem:triangle2}
  Let $G$ be a cubic graph with $k$ different (isomorphic or non-isomorphic) polyhedral embeddings of some genus $g$.
  Then there is a graph $G'$ that has $k$ different non-isomorphic polyhedral embeddings of genus $g$
\end{lemma}

\begin{proof}
  
  If the vertices of $G$ are $w_1,\dots ,w_n$, we can choose $n$ 3-connected plane cubic graphs $G_1,\dots ,G_n$ and vertices $v_1,\dots ,v_n$, so that for $1\le i \le n$,  $v_i$ is a vertex of $G_i$
  and the graphs $(G_i)_{v_i}$ are  pairwise non-isomorphic graphs  with only triangles and quadrangles as bounded faces (see e.g. Figure~\ref{fig:triangle2}).
  Now we can iteratively apply the star operation to replace each $w_i$ by $(G_i)_{v_i}$ to form a graph $G'$.
  According to Corollary~\ref{cor:equalnumber}, $G'$ also has
  $k$ polyhedral embeddings on genus $g$. As all faces in one of
  these embeddings that are not contained in a subgraph $(G_i)_{v_i}$ are
  at least hexagons, each isomorphism between two such embeddings must stabilize each subgraph as well as the set of original edges -- that is edges not in one of the $(G_i)_{v_i}$.
   As the subgraphs $(G_i)_{v_i}$ $1\le i \le n$ are equivalence classes in the sense of Proposition~\ref{note:equclass}, there are just two possible rotations inside the subgraphs and
  an embedding isomorphism between two different embeddings $I_1$, $I_2$ must therefore have an original edge with endpoints in subgraph $(G_i)_{v_i}$  and $(G_j)_{v_j}$
  where at the vertex in $(G_i)_{v_i}$ the rotation is as in  $I_1$
  and at  the vertex in $(G_j)_{v_j}$ it is different from the one in $I_1$. This means that the embedding isomorphism interchanges the other two subgraphs sharing an original edge with $(G_j)_{v_j}$  compared to $I_1$, which is impossible as
  the subgraphs are stabilized. So there is no embedding isomorphism mapping $I_1$ to $I_2$.

\end{proof}

The smallest cubic graph with 8 different (but isomorphic) embeddings on the
torus is the Heawood graph. The construction in the proof of
Lemma~\ref{lem:triangle2} with replacing all vertices is of course designed for simplicity and not
for a minimal number of vertices. Nevertheless a smallest cubic
graph with 8 non-isomorphic embeddings on the torus -- which has 22 vertices -- can be obtained from the Heawood graph by applying the star operation.

Lemma~\ref{lem:extension} now has the following direct consequences:

\begin{corollary}\label{cor:help} ~\\
 
  \begin{description}

   \item[(a)] If there is a cubic graph $G$ with $\gen(G)=g$ and the smallest genus of a polyhedral embedding is $g'$, then there is also a cubic graph with genus at most $2g$ and the smallest genus of a polyhedral embedding is $2g'$.
    
    \item [(b)]
      If there is a cubic graph $G$ with polyhedral embeddings in genus $g_1<g_2<\dots <g_m$, and a cubic graph $G'$ with polyhedral embeddings in genus  $g'_1<g'_2<\dots <g'_k$,
      then there is also a cubic graph, so that for all $1\le i \le m$ and $1\le j\le k$ there is a polyhedral embedding in genus $g_i+g'_j$.

      \item[(c)] If there is a cubic graph $G$ with $p$ (up to mirror
        images) different polyhedral embeddings in genus $g$ and a cubic graph $G'$ with $p'$ (up to mirror
        images) different polyhedral embeddings in genus $g'$, then
        there is a cubic graph with $p\cdot p'$ 
        different polyhedral embeddings in genus $g+g'$.

      \end{description}

\end{corollary}

\begin{proof}

  Part (a) is a consequence of the formula in Lemma~\ref{lem:extension} applied to  $G_v\star G_v$ for an arbitrary vertex $v$ in $G$.
  By Lemma~\ref{lem:extension} the smallest possible polyhedral embedding of $G_v\star G_v$ is in genus $2g'$ while $\gen(G_v\star G_v)\le 2~ \gen(G)$.

  Part (b) and (c) are again a direct consequence of the formula in Lemma~\ref{lem:extension} applied to  $G_v\star G'_{v'}$ for arbitrary vertices $v$ in $G$ and $v'$ in $G'$.

\end{proof}

So in order to prove Theorem~\ref{thm:existence}, we just need some graphs that can be used as initial graphs for Corollary~\ref{cor:help}.

\medskip

\subsection{Some special graphs}
~\\
We will write $G \star G'$ as a notation for a graph where arbitrary vertices $v$ in $G$ and $v'$ in $G'$ can be taken to form $G_v \star G'_{v'}$.
The Heawood graph $H$ has $8$ different (but isomorphic) embeddings on the torus. With Corollary~\ref{cor:help} (c) this implies that for each $g$ 
the graph $(\dots (H \star H)\star H)\star \dots \star H)$
where the star operation is applied $g$ times has a total of $8^g$ polyhedral embeddings in genus $g$, which implies  the first two parts of Theorem~\ref{thm:existence}~(c). The last part follows with
Lemma~\ref{lem:triangle2}.

The smallest cubic graphs that have polyhedral embeddings with
different genera have 26 vertices. There are $4$ of them and they all have one polyhedral embedding in the torus and two in genus 2.
These graphs can be found in the
graph database {\em HoG}~\cite{HoG} when searching for the keyword
\verb+pol_embed+. One of these graphs is shown with an embedding of
genus 1 in Figure~\ref{fig:cubic_pol_emb_on_two_genera_genus_1} and
with two embeddings of genus 2 in
Figure~\ref{fig:cubic_pol_emb_on_two_genera_genus_2}. There are no other polyhedral embeddings of this graph.
With Corollary~\ref{cor:help} (b) this can now be used to show that for each $k$ there is a cubic graph with embeddings in each genus $k\le g \le 2k$, which implies Theorem~\ref{thm:existence}~(b).

\begin{figure}
  	\centering
 
   \includegraphics[width=0.4\textwidth]{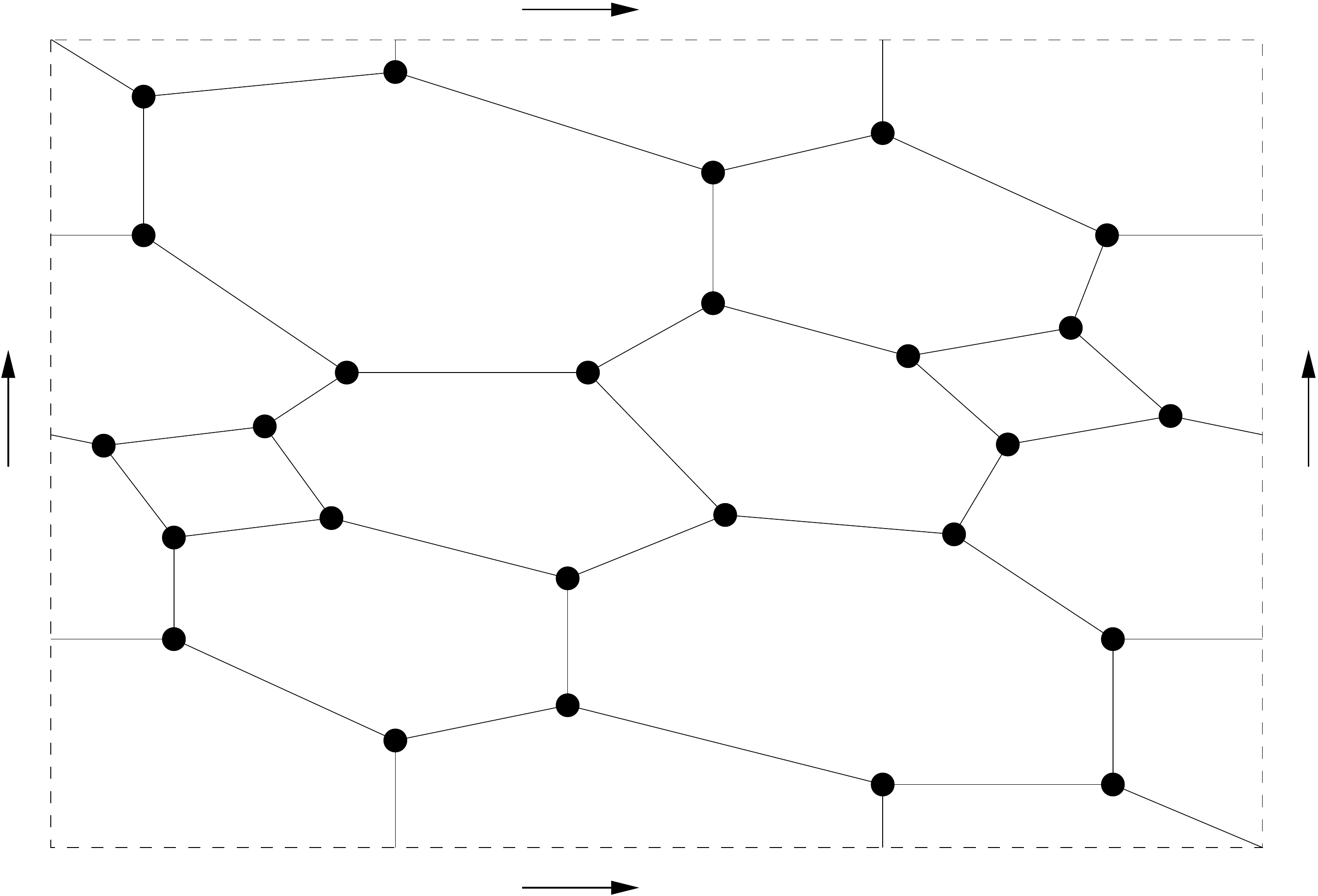}

\caption{The polyhedral embedding of genus 1 for one of the smallest cubic graphs on 26 vertices that have polyhedral embeddings of different genera. As usual, opposite sides of the rectangle need to be identified as specified by the arrows.}\label{fig:cubic_pol_emb_on_two_genera_genus_1}
\end{figure}

\begin{figure}
  	\centering
 
    \includegraphics[width=0.85\textwidth]{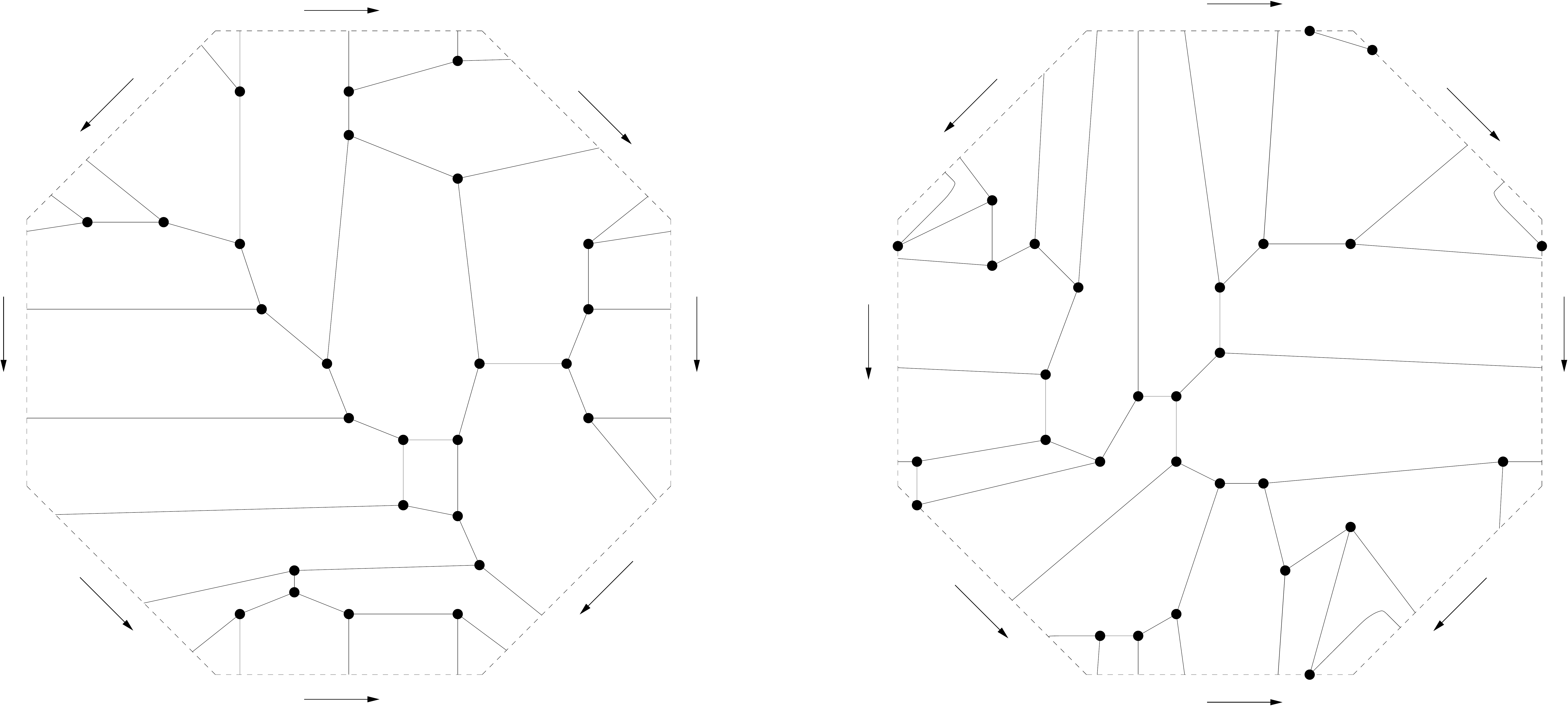}
 
\caption{The two polyhedral embeddings of genus 2 of the graph in Figure~\ref{fig:cubic_pol_emb_on_two_genera_genus_1}. Again opposite sides have to be identified.}\label{fig:cubic_pol_emb_on_two_genera_genus_2}
\end{figure}

The numbers in Table~\ref{tab:all} and Table~\ref{tab:allpolembed} together with the fact that for different embeddings of a given graph, the
number of faces decreases as the genus of the embedding
increases, suggest that it is more difficult to
get polyhedral embeddings for larger genera. One might even be tempted to
expect that if a graph has a polyhedral embedding, it also has one of minimum genus.
This is however not true and the smallest cubic counterexamples
have 28 vertices. There are four of these graphs  -- two have genus 1 and only one polyhedral embedding and that is in genus 2 and the other two have genus 2
and each has two polyhedral embeddings -- both in genus 3. One of these counterexamples is
shown in Figure~\ref{fig:cubic_pol_emb_not_min_genus_genus_2} with a
minimum genus embedding, and in
Figure~\ref{fig:cubic_pol_emb_not_min_genus_genus_3} with the unique polyhedral
embedding. Also these graphs can be found in the
graph database {\em HoG} \cite{HoG} when searching for the keyword
\verb+pol_embed+. 
With Corollary~\ref{cor:help} (a) this can now be used to show that for each $k\ge 1$ there is a cubic graph with genus $k$ whose smallest polyhedral embedding is in genus $2k$, which implies Theorem~\ref{thm:existence}~(d).

\begin{figure}
  	\centering
 
  \includegraphics[width=0.4\textwidth]{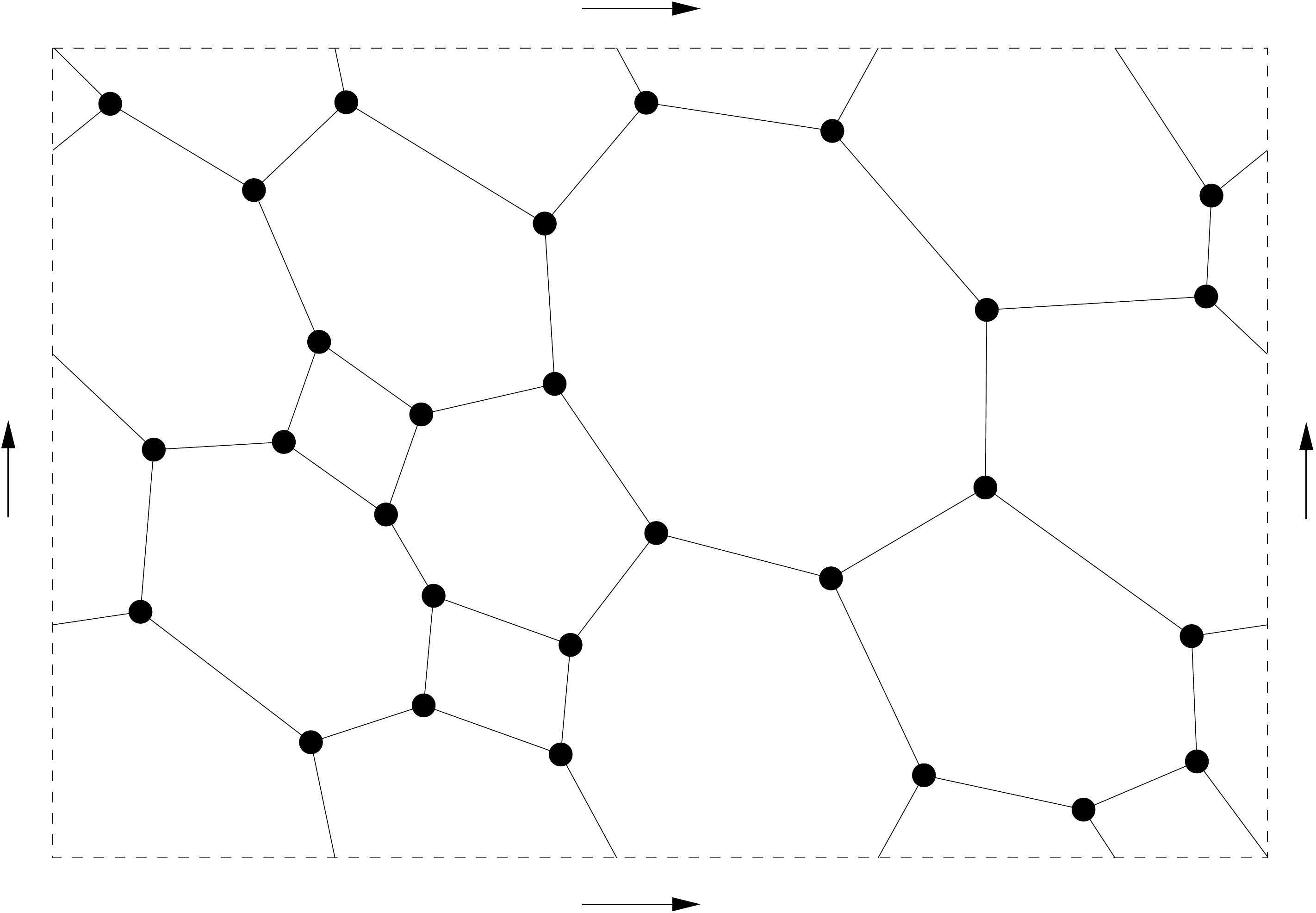}

\caption{A minimum genus embedding for one of the smallest cubic
  graphs on 28 vertices that have a polyhedral embedding, but not for
  their minimum genus. Again opposite sides have to be identified.}\label{fig:cubic_pol_emb_not_min_genus_genus_2}
\end{figure}

\begin{figure}
  	\centering
 
    \includegraphics[width=0.4\textwidth]{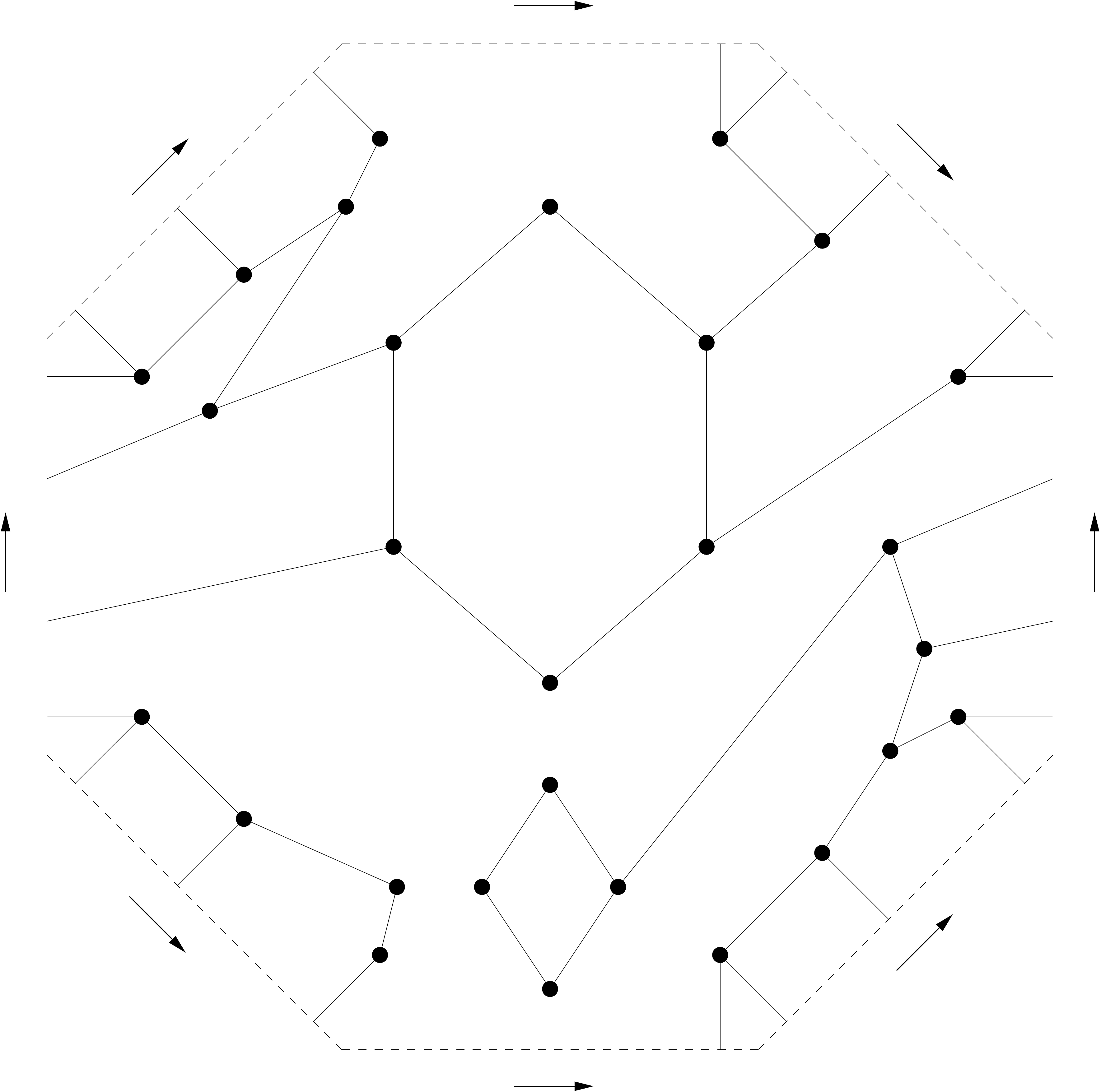}

\caption{A polyhedral embedding of the graph in Figure~\ref{fig:cubic_pol_emb_not_min_genus_genus_2}.   Again opposite sides have to be identified.}\label{fig:cubic_pol_emb_not_min_genus_genus_3}
\end{figure}

\section{Further research}

We have shown that there are cubic graphs with arbitrarily many polyhedral embeddings -- in the same genus as well as in arbitrarily many different genera --
but in the constructions, the genera of the graphs grow with the number of embeddings. In fact at least for embeddings in the same genus, the genus
of the embeddings must grow, as  -- as shown by Mohar and Robertson \cite{num_embed_bound} -- for each genus there exists an upper bound $\zeta (2g)$ on the number
of polyhedral embeddings of a (not necessarily cubic) graph in that surface of genus $g$. The factor 2 comes from the fact that they define $\zeta()$ as a function of the Euler genus.
They also show that $\zeta (2g)\ge 48^{g}$ by using $g$ copies of $K_7$ attached to a plane triangulation.
While in  \cite{num_embed_bound} the result is of an asymptotic nature, in \cite{num_embed_bound_fw4torus} Robertson, Zha, and Zhao  give $3$ as an exact bound
and characterize graphs with more than one toroidal embedding -- but under the stronger condition of facewidth at least 4. It would be very interesting to give at least a good approximation of
the upper bound in the case of polyhedral embeddings of cubic graphs and small genus $g$.

\begin{description}
\item[(a)] What is (a good approximation of) the upper bound $\zeta_3(2g)$  for the number of polyhedral embeddings of a cubic graph in a surface of genus $g$?
\end{description}

For the plane we have $\zeta_3(0)=1$ if -- as usual -- we identify mirror
images. The result proven in this paper implies that $\zeta_3(g)\ge 8^g$. We tested all cubic graphs
on up to 28 vertices. The smallest graph with a polyhedral embedding in genus 1 has 14 vertices and -- as mentioned before -- already for 14 vertices there is a graph -- the Heawood graph $H$ -- with
8 polyhedral embeddings in the torus. Note that this is in fact the dual of the graph $K_7$ embedded on the torus, which was used in  \cite{num_embed_bound}. For larger vertex numbers up
to 28 there are many more graphs with 8 polyhedral embeddings on the
torus -- but no graph with more than 8 such embeddings. For genus 2 we
have shown that there are cubic graphs with at least $8^2=64$
embeddings. The (unique) smallest cubic graph with $64$ polyhedral
embeddings in genus 2 has 26 vertices and is $H\star H$. For 28 vertices there are more
such graphs, but none with more polyhedral embeddings. As the smallest cubic
graph with a polyhedral embedding in genus 2 has 24 vertices, these
numbers are too small to draw conclusions.

We know by Corollary~\ref{cor:equalnumber} that by forming $H\star G$ for a cubic 3-connected planar graph $G$ (or even repeating this process)
gives more cubic graphs with 8 polyhedral embeddings in the torus, but the reason for the comparatively large number of embeddings is rooted in $H$, so in
a certain sense these graphs can be considered trivial. 
Due to Lemma~\ref{lem:3edgecut}, the smallest cubic graph with (at least) $8$ embeddings on the torus that cannot be obtained as a star product
must be cyclically 4-connected. We tested all cubic graphs with girth at least 4 on up to 30 vertices with $8$ embeddings on the torus and it turned out that
-- except for $H$ -- they all have a cyclic 3-edge-cut, so they
can all be obtained from $H$ by the star product.

\begin{description}
\item[(b)] What is the smallest cubic graph with at least $8$ embeddings on the torus that cannot be obtained as a repeated star product applied to the Heawood graph?
\end{description}

Note that it is also possible that such a graph does not exist. E.g. in \cite{num_embed_bound_fw4torus} the graph of the 4-by-4 grid on the torus played a special role for facewidth 4: all graphs without a unique embedding on the torus can be obtained by operations on this graph.

Though for genus 2 we do not have much data, the situation seems similar: the cubic graph with 64 embeddings in genus 2 is unique ($H^2=H\star H$). On 28 vertices there are
4 graphs with  64 embeddings in genus 2: three have girth 3 (so they must be $H^2\star K_4$) and the other has girth 6, but a cyclic 3-edge-cut, so it must be
$(H\star K_4)_v \star H$ with the vertex $v$ in the triangle of $H\star K_4$.

A question asking for a similar result to the main result in \cite{num_embed_bound}, but not fixing the genus of the embedding, but of the graph, is completely open:

\begin{description}
\item[(c)]  Is there an upper bound $f_a(g)$  for the number of polyhedral embeddings of a cubic graph with genus $g$?
\end{description}

Again by Theorem~\ref{thm:polyhedra} $f_a(0)=1$, but the fact that
already toroidal graphs can have polyhedral embeddings in the torus as
well as in an arbitrarily large genus, suggests that not restricting
the genus of the embedding might lead to arbitrarily many polyhedral
embeddings already for the torus. Nevertheless, we know no proof of
this and up to 28 vertices all
graphs $G$ with the largest number of polyhedral embeddings for a fixed value of $\gen(G)$ have all
embeddings in the same genus -- which is in fact $\gen(G)$.

\medskip

For the difference between the smallest genus that allows a polyhedral
embedding and the genus of the graph, we have a similar situation: for
the graphs constructed, the difference is only large for graphs that
have a large genus, while even toroidal graphs can have embeddings in arbitrarily large genus. This suggests the following question:

\begin{description}
\item[(d)] Is there an upper bound $f_p(g)$ on the minimum genus of a polyhedral embedding of a cubic graph $G$ with $\gen(G)=g$ that has a polyhedral embedding in some genus? 
\end{description}

For the plane we have $f_p(0)=0$, but already for $g=1$ the question is open.


\end{document}